\documentclass[12pt,a4paper,reqno]{amsart}
\usepackage{amsfonts,amsthm,amsmath,amscd}
\usepackage[utf8]{inputenc}
\usepackage{t1enc}
\usepackage[mathscr]{eucal}
\usepackage{indentfirst}
\usepackage{graphicx,graphics,pict2e}
\usepackage{epic}
\usepackage{lipsum}
\usepackage{stmaryrd}
\usepackage{authblk}
\usepackage{subcaption}
\usepackage[symbol]{footmisc}
\numberwithin{equation}{section}
\usepackage[margin=2.9cm]{geometry}
\usepackage{epstopdf}
\RequirePackage{doi}
\usepackage{hyperref}
\allowdisplaybreaks
\usepackage{cite}
\usepackage{verbatim,amssymb,amsfonts}
\usepackage{mathrsfs}
\usepackage{mathtools}
\usepackage{tikz-cd}
\usepackage{indentfirst}
\usepackage{slashed}
\usepackage{thmtools}
\usepackage{setspace}
\usepackage{enumerate}
\usepackage{accents}

\theoremstyle{plain}
\newtheorem{theorem}{Theorem}[section]
\newtheorem{lemma}[theorem]{Lemma}

\theoremstyle{definition}
\newtheorem{definition}[theorem]{Definition}
\newtheorem{remark}[theorem]{Remark}

\DeclarePairedDelimiterX{\inp}[2]{\langle}{\rangle}{#1, #2}

\renewcommand{\thefootnote}{\fnsymbol{footnote}}

\makeatletter
\g@addto@macro{\endabstract}{\@setabstract}
\newcommand{\authorfootnotes}{\renewcommand\thefootnote{\@fnsymbol\c@footnote}}%
\makeatother

\makeatletter
\@namedef{subjclassname@2020}{%
  \textup{2020} Mathematics Subject Classification}
\makeatother

\title[]{Some New Results on the Seidel Energy of Graphs with Self-Loops}

\begin{document}

\begin{center}
    \vspace{-1cm}
	\maketitle

	\normalsize
    \authorfootnotes
    Kalpesh M. Popat\textsuperscript{1},
    Irena M. Jovanovi\' c \textsuperscript{2},
	\par \smallskip

\textsuperscript{1}
        \small{Saurashtra University, Rajkot, Gujarat, India}

       \textsuperscript{2}
        \small{School of Computing, Union University, Belgrade, Serbia} \par \bigskip

\end{center}

\address{Saurashtra University, Rajkot, Gujarat, India}
\email{kalpeshmpopat@gmail.com}

\address{School of Computing, Union University, Belgrade, Serbia}
\email{irenaire@gmail.com}


\begin{abstract}
Harshitha et al. recently introduced Seidel energy of graphs with self-loops. In this paper, we extend some of their results by giving a necessary and sufficient condition for the Seidel energy of a looped graph to be equal to the Seidel energy of its underlying graph. We also consider Seidel energy of the union of certain graphs, and show that graph operations complement and Seidel switching preserve Seidel energy in the looped setting.    

\vspace{1.0em}
\noindent \textbf{Key words:} self-loop graph, Seidel matrix, Seidel energy, Seidel switching, complemet, union.   

\smallskip
\noindent \textbf{2020 Mathematics Subject Classification:} 05C50

\end{abstract}

\section{Introduction}

Let $G=(V(G),E(G))$ be a finite simple graph with the vertex set $V(G)$ and the edge set $E(G)$. In case $|V(G)|=n$ and $|E(G)|=m$, we say that $G$ is the graph of order $n$ and size $m$. 

Let $W\subseteq V(G)$ and $|W|=\sigma$, $0\leq\sigma\leq n$, and let us denote by $G_W$ the \emph{self-loop graph} of $G$ at $W$. More precisely, $G_W$ is the graph of order $n$ obtained from $G$ by attaching a self-loop at each vertex from the set $W$. In the literature so far, there are
some studies of the spectrum and related properties of graphs with self-loops, see, for example, \cite{cesar}, \cite{leslie}, \cite{gutman2022}, \cite{jovanovic2023}, \cite{zagrebloop}, \cite{loopbound1}. In this paper, we are interested in spectral properties of these graphs connected to their energy with respect to the Seidel matrix.

The \emph{Seidel matrix} $S=S(G)=[s_{ij}]_{n\times n}$ of the graph $G$ with $n$ vertices, originally introduced in \cite{seidel}, is the real symmetric square matrix of order $n$ whose $(i,j)$-th entry $s_{ij}$ is equal to $-1$, if the corresponding vertices $v_i$ and $v_j$, $v_i,v_j\in V(G)$ and $i\neq j$, are adjacent in $G$, to $1$, if vertices $v_i$ and $v_j$ are non-adjacent in $G$, and $s_{ii}=0$, for all $i=1,2,\ldots, n$. Graphs with self-loops with respect to the Seidel matrix have been studied in \cite{harshitha2024}, where the following definition for the Seidel matrix $S(G_W)$ of the self-loop graph $G_W$ of the given graph $G$ has been used:
$$S(G_W)=S(G)-\mathfrak{I}_W.$$ 
Here, $\mathfrak{I}_W$ is the diagonal matrix of order $n$ with exactly $\sigma$ ones on the main diagonal corresponding to $W$ and all other entries equal to zero. The eigenvalues $\theta_1(G)\geq\theta_2(G)\geq\cdots\geq\theta_n(G)$ ($\theta_1(G_W)\geq\theta_2(G_W)\geq\cdots\geq\theta_n(G_W)$) of $S(G)$ ($S(G_W)$) are the \emph{eigenvalues} of $G$ ($G_W$). The multiset of the eigenvalues of $G$ ($G_W$) form the \emph{spectrum}, in the notation $spec(G)$ ($spec(G_W)$), of $G$ ($G_W$). We say that (self-loop) graphs of the same order are \emph{cospectral} if their spectra coincide.

The study of the graph energy, originally with respect to the adjacency matrix, was initiated by Ivan Gutman in the 1970s \cite{Gutman}. Nowadays, it is one of the most intensively studied graph invariants \cite{li2012}. Traditionally, the theory of the graph energy is focused on simple graphs and the adjacency matrix. But, in recent years, the classical energy concepts are extended to graphs with self-loops and various graph matrices \cite{AkbJovLim}, \cite{sloop3}, \cite{anchan2023}, \cite{popat2023}, \cite{popat2024}, \cite{preetha2023}. In this paper, we consider the energy of (self-loop) graphs regarding the Seidel matrix. 

The \emph{Seidel energy} $\mathcal{SE}(G)$ of the graph $G$ is defined in \cite{haemers} as the sum of the absolute values of the eigenvalues of the Seidel matrix $S(G)$ of the graph $G$, i.e. $$\mathcal{SE}(G)=\sum\limits_{i=1}^{n} |\theta_i(G)|.$$ 
The Seidel energy was studied intensively for loop-less graphs \cite{akbari2020}, \cite{einollahzadeh2022}, \cite{oboudi2016}, \cite{oboudi2023}, \cite{popat2019}. 

The \emph{Seidel energy} of the self-loop graph $G_W$ was introduced and considered in \cite{harshitha2024} by the following definition:
$$\mathcal{SE}(G_W)=\sum\limits_{i=1}^{n} \left|\theta_i(G_W)+\frac{\sigma}{n}\right|.$$
Harshitha et al. \cite{harshitha2024} derived several fundamental results connected with the Seidel energy of self-loop graphs, including bounds on $\mathcal{SE}(G_W)$, relations between $\mathcal{SE}(G_W)$ and $\mathcal{SE}(G)$, as well as explicit formulae for the Seidel energy of the complete and complete bipartite graphs with self-loops. In this paper, we extend some of their results. Namely, we give a necessary and sufficient condition for the equality of the Seidel energies of a looped graph and its underlying graph, consider Seidel energy of the union of certain graphs, and show that graph operations complement and Seidel switching preserve Seidel energy in the looped setting.

\section{Preliminaries}

In this section, we will recall some basic notation and fundamental statements which will be used along the paper.

Let $G$ be a simple graph of order $n$, and let $G_W$ be the self-loop graph of $G$, where $W\subseteq V(G)$ and $|W|=\sigma$, $0\leq\sigma\leq n$. 

The concept of the matrix energy was established by Nikiforov \cite{nikiforov}. Let $M$ be a real matrix of order $n\times m$, with singular values $s_i(M)$, $i=1,2,\ldots,q$. The \emph{energy} $\mathcal{E}(M)$ \emph{of the matrix} $M$ is defined as: 
$$\mathcal{E}(M)=\sum\limits_{i=1}^{q} s_i(M),$$ 
where $q\leq\min\{n,m\}$. If $M$ is a real symmetric square matrix of order $n$, its singular values $s_i(M)$, $i=1,2,\ldots,n$, are equal to the absolute values of its eigenvalues. Therefore, we have 
$$\mathcal{SE}(G)=\mathcal{E}(S(G))=\sum\limits_{i=1}^{n} s_i(S(G))=\sum\limits_{i=1}^{n} |\theta_i(G)|,$$
and
$$\mathcal{SE}(G_W)=\mathcal{E}\left(S(G_W)+\frac{\sigma}{n}I_n\right)=\sum\limits_{i=1}^{n} s_i\left(S(G_W)+\frac{\sigma}{n}I_n\right)=\sum\limits_{i=1}^{n} \left|\theta_i(G_W)+\frac{\sigma}{n}\right|,$$
where $I_n$ stands for the identity matrix of order $n$. For the all-ones square matrix of order $n$ we will use the label $J_n$. Given the previous remarks, the following statements are of importance in the subsequent sections of the paper:

\begin{theorem}\cite{fan}\label{tfan}
Let $A, B$ and $C$ be square matrices of order $n$, such that $C=A+B$. Then
$$\sum\limits_{i=1}^{n} s_i(C)\leq \sum\limits_{i=1}^{n} s_i(A) + \sum\limits_{i=1}^{n} s_i(B).$$
Equality holds if and only if there exists an orthogonal matrix $P$, such that $PA$ and $PB$ are both positive semi-definite.
\end{theorem}

\begin{lemma}\cite{hornjohnson}\label{lemmaaux}
Let $A=[a_{ij}]_{n\times n}$ be a positive semi-definite matrix of order $n$ such that $a_{ii}=0$ for some $1\le i\le n$. Then $a_{ij}=a_{ji}=0$ for all $1\le j\le n$.
\end{lemma}

\begin{lemma}\cite{fiedler}\label{lfiedler}
Let $A$ be a symmetric $m\times m$ matrix with eigenvalues $\alpha_1,\ldots, \alpha_m$, let $u$, $\|u\|=1$, be a unit eigenvector corresponding to $\alpha_1$; let $B$ be a symmetric $n\times n$ matrix with eigenvalues $\beta_1,\ldots, \beta_n$, let $v$, $\|v\|=1$, be a unit eigenvector corresponding to $\beta_1$. 
Then for any $\rho$, the matrix
$$C=\left(
      \begin{array}{cc}
        A & \rho\,uv^T \\
        \rho\,vu^T & B \\
      \end{array}
    \right)
$$
has eigenvalues $\alpha_2,\ldots, \alpha_m,\beta_2,\ldots, \beta_n, \gamma_1, \gamma_2$, where $\gamma_1, \gamma_2$ are eigenvalues of the matrix
$$\widehat{C}=\left(
                \begin{array}{cc}
                  \alpha_1 & \rho \\
                  \rho & \beta_1 \\
                \end{array}
              \right).
$$
\end{lemma}

Let us remind that $G$ is the \emph{$r$-regular graph} if all its vertices are of the same degree equal to $r$. The graph $G \cup H$ means the \emph{disjoint union} of graphs $G$ and $H$. The \emph{join} $G\nabla H$ of disjoint graphs $G$ and $H$ is the graph obtained from $G \cup H$ by joining each vertex of $G$ to each vertex of $H$. The analogous definitions also apply in the case graphs have loops.

\section{A Relationship Between Seidel Energies of a Looped Graph and its Underlying Graph}

Let $G$ be a graph of order $n$, and let $W\subseteq V(G)$, where $|W|=\sigma$. Harshitha et al.~\cite{harshitha2024} have proved that $\mathcal{SE}(G_W) = \mathcal{SE}(G)$ \textit{if} $\sigma = 0$ or $\sigma = n$. In this section, we continue to study the relationship between $\mathcal{SE}(G)$ and $\mathcal{SE}(G_W)$ by considering the converse of the mentioned statement.

\begin{theorem}
Let $G$ be a non-empty graph on $n$ vertices and let $W\subseteq V(G)$ with $|W|=\sigma$,
$0\le\sigma\le n$. Then
\begin{equation}\label{relations}
\mathcal{SE}(G)-2\sigma\!\left(1-\frac{\sigma}{n}\right)
\;\le\;
\mathcal{SE}(G_W)
\;\le\;
\mathcal{SE}(G)+2\sigma\!\left(1-\frac{\sigma}{n}\right),
\end{equation}
where the left-hand side and the right-hand side inequality is attained if and only if $\sigma=0$ or $\sigma=n$.
\end{theorem}

\begin{proof}
It holds
\begin{equation}\label{relation1}
S(G_W)+\frac{\sigma}{n}I_n = S(G)+\left(
                                      \begin{array}{cc}
                                        \left(\frac{\sigma}{n}-1\right)I_\sigma & O \\
                                        O & \frac{\sigma}{n}I_{n-\sigma} \\
                                      \end{array}
                                    \right) = S(G)+E,
\end{equation}
where 
$$E=\left(
                                      \begin{array}{cc}
                                        \left(\frac{\sigma}{n}-1\right)I_\sigma & O \\
                                        O & \frac{\sigma}{n}I_{n-\sigma} \\
                                      \end{array}
                                    \right).$$ 
Since $\sum\limits_{i=1}^{n} s_i(E)=\sum\limits_{i=1}^{n} s_i(-E)=2\sigma\left(1-\frac{\sigma}{n}\right)$, by applying Theorem \ref{tfan} to (\ref{relation1}), we obtain the right-hand side inequality in (\ref{relations}). As (\ref{relation1}) is equivalent to
\begin{equation}\label{relation2}
S(G)=S(G_W)+\frac{\sigma}{n}I_n-E,
\end{equation}
by applying Theorem \ref{tfan} to (\ref{relation2}), we obtain the left-hand side inequality in (\ref{relations}).

Let us suppose that the right-hand side equality holds in (\ref{relations}), and that $0<\sigma<n$. According to Theorem \ref{tfan} and the equality (\ref{relation1}), there exists an orthogonal matrix $P$ such that $PS(G)$ and $PE$ are both positive semi-definite. Since $S(G)$ is symmetric, we may assume that $P$ is also symmetric. Let
                                     $$P=\left(
                                               \begin{array}{cc}
                                                 P_{11} & P_{12} \\
                                                 P_{21} & P_{22} \\
                                               \end{array}
                                             \right),
                                    $$
where $P_{11}$ and $P_{22}$ are square matrices of order $\sigma$ and $n-\sigma$, respectively. Since the matrix
$$PE=P
\begin{pmatrix}
\left(\frac{\sigma}{n}-1\right)I_\sigma & O\\
O & \frac{\sigma}{n}I_{n-\sigma}
\end{pmatrix}
=
\begin{pmatrix}
\left(\frac{\sigma}{n}-1\right)P_{11} &
\frac{\sigma}{n}P_{12} \\[4pt]
\left(\frac{\sigma}{n}-1\right)P_{21} &
\frac{\sigma}{n}P_{22}
\end{pmatrix}$$ 
is positive semi-definite, it must be symmetric, and therefore 
\begin{equation}\label{auxiliary}
\frac{\sigma}{n}P_{12}
=
\left(\frac{\sigma}{n}-1\right)P_{21}^{T}.
\end{equation}
On the other hand, since $P$ is symmetric, $P_{12}=P_{21}^{T}$, which, together with (\ref{auxiliary}), gives
$$\frac{\sigma}{n}P_{12}=\left(\frac{\sigma}{n}-1\right)P_{12},$$
wherefrom it follows $P_{12}=O$, and hence $P_{21}=O$. Therefore, $$P=\left(
                                                                       \begin{array}{cc}
                                                                         P_{11} & O \\
                                                                         O & P_{22} \\
                                                                       \end{array}
                                                                     \right)\qquad{\rm and}\qquad PE=\left(
                                                                                                                   \begin{array}{cc}
                                                                                                                     \left(\frac{\sigma}{n}-1\right)P_{11} & O \\
                                                                                                                     O & \frac{\sigma}{n}P_{22} \\
                                                                                                                   \end{array}
                                                                                                                 \right)
.$$

Since $P$ is orthogonal matrix, we have $P_{11}^2=I_\sigma$ and $P_{22}^2=I_{n-\sigma}$. Since, $PE$ is the positive semi-definite matrix and $0<\sigma<n$, $P_{11}$ is negative semi-definite, while $P_{22}$ is positive semi-definite. Let us consider the matrix $P_{22}$. Since it is real and positive semi-definite, it is symmetric, i.e. $P_{22}^T=P_{22}$, and therefore it is diagonalizable. The equation $P_{22}^2=I_{n-\sigma}$ implies the minimal polynomial of $P_{22}$ divides the polynomial of the form $(\lambda-1)(\lambda+1)$, which means that the eigenvalues of $P_{22}$ belong to the set $\{-1,1\}$. Since $P_{22}$ is positive semi-definite, all its eigenvalues must be non-negative, i.e. equal to $1$. Therefore, $P_{22}$ is the identity matrix, i.e. $P_{22}=I_{n-\sigma}$. The similar reasoning can be applied to the matrix $P_{11}$, wherefrom it follows $P_{11}=-I_{\sigma}$. Hence, $$P=\left(
                                                                            \begin{array}{cc}
                                                                              -I_{\sigma} & O \\
                                                                              O & I_{n-\sigma} \\
                                                                            \end{array}
                                                                          \right)
.$$

As $$S(G)=\left(\begin{array}{cc}
                S(G[W]) & M \\
                M^T & S(G[V(G)\setminus W]) \\
                \end{array}
         \right),$$ 
where $G[W]$ and $G[V(G)\setminus W]$ are the subgraphs of $G$ induced by the vertex sets $W$ and $V(G)\setminus W$, respectively, we have         
 
$$PS(G)=\left(\begin{array}{cc}
              -S(G[W]) & -M \\
               M^T & S(G[V(G)\setminus W]) \\
               \end{array}
               \right).$$
Let us notice that the block matrices $-S(G[W])$ and $S(G[V(G)\setminus W])$ have zero diagonal. Since $PS(G)$ is positive semi-definite with zero diagonal, from Lemma \ref{lemmaaux}, it follows $S(G)=O$, which is impossible.

If $\sigma=0$ or $\sigma=n$, the right-hand side equality follows immediately.

Regarding the left-hand side equality, let us suppose it holds in (\ref{relations}), and let $0<\sigma<n$. Taking into account equality (\ref{relation2}) and using Theorem \ref{tfan}, it follows there exists an orthogonal matrix $P$ such that the matrices $P\left(S(G_W)+\frac{\sigma}{n}I_n\right)$ and $P\left(-E\right)$ are both positive semi-definite. As before, let 
$$P=\left(
                                               \begin{array}{cc}
                                                 P_{11} & P_{12} \\
                                                 P_{21} & P_{22} \\
                                               \end{array}
                                             \right).
                                    $$
In a similar way as it is in the case of the proof of the right-hand side inequality, we obtain $P_{12}=P_{21}=O$, and $P_{11}=I_{\sigma}$ and $P_{22}=-I_{n-\sigma}$. Further we have
$$P\left(S(G_W)+\frac{\sigma}{n}I_n\right)=\left(
                                             \begin{array}{cc}
                                               S(G[W])-\left(1-\frac{\sigma}{n}\right)I_{\sigma} & M \\
                                               -M^T & -S(G[V(G)\setminus W])-\frac{\sigma}{n}I_{n-\sigma} \\
                                             \end{array}
                                           \right),
$$
where, as before, $G[W]$ and $G[V(G)\setminus W]$ are the subgraphs of $G$ induced by the vertex sets $W$ and $V(G)\setminus W$, respectively. Since $0<\sigma<n$, all diagonal entries of the matrix $P\left(S(G_W)+\frac{\sigma}{n}I_n\right)$ are negative, which is the contradiction with the assumption that this matrix is positive semi-definite. Therefore, $\sigma=0$ or $\sigma=n$.

If $\sigma=0$ or $\sigma=n$, the left-hand side equality follows immediately.

\end{proof}

\section{Seidel Energy of Looped Graphs with Respect to Complement and Seidel Switching}

It can be easily verified that the Seidel energy of graphs is invariant under Seidel switching and complement. In this section, we will consider these two graph ope\-ra\-ti\-ons in the case of graphs with self-loops, and show that they also preserve the Seidel energy in such a setting.

Depending on how the complement can affect the existence of loops in a graph, there are two natural conventions for defining this graph operation in case of looped graphs. We choose the one that complements loops:

\begin{definition}\label{complement}
Let $G$ be a graph, and let $\emptyset \neq W \subseteq V(G)$. The \emph{complement} $\overline{G_W}$ of the self-loop graph $G_W$ is the self-loop graph on the same vertex set as $G_W$, where any two distinct vertices are adjacent if and only if they are not adjacent in $G_W$, and where a vertex has a loop attached if and only if it does not have a loop in $G_W$.
\end{definition}

Definition \ref{complement} ensures the following identity:
\begin{equation}\label{complementIdentity}
S(\overline{G_W}) = -S(G_W) - I_n.
\end{equation}

\begin{theorem}
Let $G$ be a graph of order $n$, and let $\emptyset \neq W \subseteq V(G)$. Then
$$\mathcal{SE}(\overline{G_W}) = \mathcal{SE}(G_W).$$
\end{theorem}

\begin{proof}
From (\ref{complementIdentity}), it follows $\theta_i(\overline{G_W}) = -\theta_i(G_W) - 1$, for $i = 1, 2, \ldots, n$. So, if $\sigma$ is the number of loops in $G_W$, and $\overline{\sigma}$ in $\overline{G_W}$, then
\[
\mathcal{SE}(\overline{G_W}) = \sum_{i=1}^{n} \left| \theta_i(\overline{G_W}) + \frac{\overline{\sigma}}{n} \right|
= \sum_{i=1}^{n} \left| -\theta_i(G_W) - 1 + 1 - \frac{\sigma}{n} \right|
= \mathcal{SE}(G_W).
\]
\end{proof}

Now, we naturally extend the notion of Seidel switching to graphs with self-loops, ensuring consistency with the above definition of the complement.

\begin{definition}
Let $G$ be an arbitrary graph, and let $\emptyset \neq W, X \subseteq V(G)$. The \emph{looped Seidel switching} with respect to $X$ is the self-loop graph $(G_W)^{X}$ obtained from $G_W$ by complementing all adjacencies (i.e. by changing edges with non-edges, and vice versa) between $X$ and $V(G_W) \setminus X$, while leaving adjacencies inside $X$ and $V(G_W) \setminus X$ unchanged.
\end{definition}

Note that the presence of loops in the graph is not affected by the looped Seidel switching.

\begin{theorem}\label{switching-thm}
Let $G$ be a graph of order $n$, and let $\emptyset \neq W, X \subseteq V(G)$. It holds
$$
spec (S((G_W)^X)) = spec (S(G_W))
\quad\text{and}\quad
\mathcal{SE}((G_W)^X) = \mathcal{SE}(G_W).
$$
\end{theorem}

\begin{proof}
It is a matter of routine to verify that the Seidel matrices of $G_W$ and $(G_W)^X$ are similar. Precisely,
\[
S((G_W)^X) = D_X S(G_W) D_X,
\]
where $D_X = [d_{ij}^X]$ is the diagonal matrix whose diagonal entries are given by
\[
d_{ii}^X =
\begin{cases}
-1, & v_i \in X,\\[2pt]
\;\;1, & v_i \notin X,
\end{cases}
\]
and $D_X^{-1} = D_X = D_X^{\top}$. Hence, the eigenvalues of $S(G_W)$ and $S((G_W)^X)$ coincide, implying equality of the corresponding Seidel energies.
\end{proof}

\section{Seidel Energy of Certain Looped Graphs with Respect to Union and Join}

In this section, we compute the Seidel energy of the disjoint union of a regular graph and its looped copy, and remark that the Seidel energy of the join of these two graphs has the same value.

\begin{theorem}\label{thm:union}
Let $G$ be a $r$-regular graph on $n$ vertices whose eigenvalues with respect to the Seidel matrix are $\theta_i(G)$, $i=1,2,\ldots,n$, and let $\mathcal{U} = G \cup G'$, where $G'$ is a copy of $G$ with a self-loop at each vertex. If $|\theta_i(G)| \ge \tfrac{1}{2}$, for $i = 2,\ldots,n$, then
\[
\mathcal{SE}(\mathcal{U}) = 2\,\mathcal{SE}(G) + 2\!\left( \max\{|\theta_1(G)|,\, R\} - |\theta_1(G)| \right),
\]
where \( R = \sqrt{\,n^2 + \tfrac{1}{4}} \).  
In particular, if \( |\theta_1(G)| \le R \), then
\[
\mathcal{SE}(\mathcal{U}) = 2\sqrt{n^2 + \tfrac{1}{4}} + 2\sum_{i=2}^n |\theta_i(G)|.
\]
\end{theorem}

\begin{proof}
The graph $\mathcal{U}$ has $2n$ vertices and $n$ loops. So, let us denote
$$\mathcal{M}=S(\mathcal{U})+\frac{1}{2}I_{2n} = \left(
                  \begin{array}{cc}
                    S(G) + \frac{1}{2}I_n & J_n \\
                    J_n & S(G) - \frac{1}{2}I_n \\
                  \end{array}
                \right)=\left(
                          \begin{array}{cc}
                            A & J_n \\
                            J_n & B \\
                          \end{array}
                        \right),
$$
where $A=S(G)+\frac{1}{2}I_n$ and $B=S(G)-\frac{1}{2}I_n$. 

Matrices $A$ and $B$ are symmetric and square of order $n$, with constant row sums equal to $n-2r-\frac{1}{2}$ and $n-2r-\frac{3}{2}$, respectively. Let us denote by $\alpha_i(A)$ and $\alpha_i(B)$, $i=1,2,\ldots,n$, the eigenvalues of these matrices, respectively. Therefore, the all-ones vector $\textbf{1}=(\underbrace{1,\ldots,1}_n)^T$ is the eigenvector of $A$ and $B$ corresponding to their eigenvalues $\alpha_1(A)=n-2r-\frac{1}{2}$ and $\alpha_1(B)=n-2r-\frac{3}{2}$, respectively. Let $u=\frac{1}{\sqrt{n}}\,\textbf{1}=\left(\frac{1}{\sqrt{n}},\ldots,\frac{1}{\sqrt{n}}\right)^T$ and $v=\frac{1}{\sqrt{n}}\,\textbf{1}=\left(\frac{1}{\sqrt{n}},\ldots,\frac{1}{\sqrt{n}}\right)^T$ be the unit eigenvectors corresponding to $\alpha_1(A)$ and $\alpha_1(B)$, respectively. Then it holds: $uv^T=\frac{1}{n}\,J_n$ and $vu^T=\frac{1}{n}\,J_n$. So, by calling Lemma \ref{lfiedler} and by setting $\rho$ to be equal to $n$, we obtain the eigenvalues of the matrix $\mathcal{M}$: $\alpha_2(A),\ldots, \alpha_n(A)$, $\alpha_2(B),\ldots, \alpha_n(B)$, and two eigenvalues $\gamma_1$ and $\gamma_2$ which are the eigenvalues of the matrix:
$$\mathcal{\widehat{M}}=\left(
                          \begin{array}{cc}
                            n-2r-\frac{1}{2} & n \\
                            n & n-2r-\frac{3}{2} \\
                          \end{array}
                        \right),
$$
i.e. the roots of the corresponding characteristic polynomial 
$$\phi_{\mathcal{\widehat{M}}}(x)=\left(x-n+2r+\frac{1}{2}\right)\left(x-n+2r+\frac{3}{2}\right)-n^2.$$ 
More precisely, the eigenvalues of the matrix $\mathcal{M}$ are: $n-1-2r\pm\sqrt{n^2+\frac{1}{4}}$, i.e. (see, for example, \cite{bro&hae}) $\theta_1(G)\pm R$, and $\theta_i(G)\pm\frac{1}{2}$, for $i = 2,\ldots,n$. Therefore,
$$\mathcal{SE}(\mathcal{U})=\mathcal{E}(\mathcal{M})=|\theta_1(G) + R| + |\theta_1(G) - R| + \sum_{i=2}^n \Big(|\theta_i(G) + \tfrac12| + |\theta_i(G) - \tfrac12|\Big).$$

Using the identity \( |a + \delta| + |a - \delta| = 2\max\{|a|, \delta\} \), for reals $a$ and $\delta$, and $\delta \ge 0$, we find
\[
\mathcal{SE}(\mathcal{U})
= 2\max\{|\theta_1(G)|, R\} + 2\sum_{i=2}^n\max\{|\theta_i(G)|, \tfrac12\}.
\]
According to the assumption, $|\theta_i(G)| \ge \tfrac12$, for $i = 2,\ldots,n$, which means that $\max\{|\theta_i(G)|, \tfrac12\} = |\theta_i(G)|$, $i = 2,\ldots,n$, and since $\mathcal{SE}(G) = \sum\limits_{i=1}^n |\theta_i(G)|$,
it follows
\[
\mathcal{SE}(\mathcal{U})
= 2\,\mathcal{SE}(G) + 2\big(\max\{|\theta_1(G)|, R\} - |\theta_1(G)|\big).
\]
In particular, if $|\theta_1(G)| \le R$, then $|\theta_1(G)+R| + |\theta_1(G)-R| = 2R$, and hence
\[
\mathcal{SE}(\mathcal{U}) = 2R + 2\sum_{i=2}^n |\theta_i(G)|
= 2\sqrt{n^2 + \tfrac14} + 2\sum_{i=2}^n |\theta_i(G)|.
\]
This completes the proof.
\end{proof}

\begin{remark}
The graph with self-loops $\mathcal{J} = G \nabla G'$ is obtained from the self-loop graph $\mathcal{U} = G \cup G'$ by looped Seidel switching with respect to the vertex set $V(G)$, i.e. $V(G')$. Therefore, according to Theorem \ref{switching-thm}, $\mathcal{U}$ and $\mathcal{J}$ are cospectral with respect to the Seidel matrix, which implies that their Seidel energies are equal.    
\end{remark}

\noindent\textbf{Acknowledgement.} Author I. M. Jovanovi\'c is supported by Project No. H20240855 of the Ministry of Human Resources and Social Security of P. R. China.


\begin{thebibliography}{99}

\bibitem{cesar} C. O. Aguilar, Anti-regular graphs with loops and their spectrum, arXiv:1805.08287 [math.CO], https://doi.org/10.48550/arXiv.1805.08287.

\bibitem{leslie} C. Bozeman, A. Ellsworth, L. Hogben, J. Chin-Hung Lin, G. Maurer, K. Nowak, A. Rodriguez, J. Strickland, Minimum rank of graphs with loops, \emph{Electronic Journal of Linear Algebra}, \textbf{27} (2014) 907--934.

\bibitem{gutman2022} I. Gutman, I. Red\v{z}epovi\'{c}, B. Furtula, A. Sahal, Energy of graphs with self-loops, \emph{MATCH Commun. Math. Comput. Chem.}, \textbf{87} (2022) 645--652.

\bibitem{jovanovic2023} I. M. Jovanovi\'{c}, E. Zogi\'{c}, E. Glogi\'{c}, On the conjecture related to the energy of graphs with self-loops, \emph{MATCH Commun. Math. Comput. Chem.}, \textbf{89} (2) (2023) 479--488.

\bibitem{zagrebloop} S. S. Shetty, A. K. Bhat, On the first Zagreb index of graphs with self-loops, \emph{AKCE Int. J. Graphs Comb.} \textbf{20} (3) (2023) 326--331. 

\bibitem{loopbound1} J. Liu, Y. Chen, D. Dimitrov, J. Chen, New bounds on the energy of graphs with self–loops, \emph{MATCH Commun. Math. Comput. Chem.} \textbf{91} (2024) 779--796. 

\bibitem{seidel} J. H. van Lint, J. J. Seidel, Equilateral point sets in elliptic geometry, \emph{Indagationes Ma\-the\-ma\-ti\-cae}, \textbf{28} (1966) 335--348. 

\bibitem{harshitha2024} A. Harshitha, S. D’ Souza, S. Nayak, I. Gutman, Seidel energy of a graph with self-loops, \emph{Communications in Combinatorics and Optimization} (2024) doi: 10.22049/cco.2024.29576.2062.

\bibitem{Gutman} I. Gutman, The energy of a graph, \emph{Ber. Math. -Statist. Sekt. Forschungsz. Graz}, {\bf 103} (1978) 1--22.  

\bibitem{li2012} X. Li, Y. Shi, I. Gutman, \emph{Graph Energy}, Springer, New York, 2012. 

\bibitem{AkbJovLim} S. Akbari, I. M. Jovanovi\' c, J. Lim, Line graphs and Nordhaus-Gaddum type bounds for self-loop graphs, \emph{Bull. Malays. Math. Sci. Soc.} (2024) 47:117.     

\bibitem{sloop3} S. Akbari, H. Al Menderj, M. H. Ang, J. Lim, Z. C. Ng, Some results on spectrum and energy of graphs with loops, \emph{Bull. Malays. Math. Sci. Soc.} (2023) 46:94. 

\bibitem{anchan2023} D. V. Anchan, S. D’Souza, H. J. Gowtham, P. G. Bhat, Laplacian energy of a graph with self-loops, \emph{MATCH Commun. Math. Comput. Chem.}, \textbf{90} (1) (2023) 247--258.

\bibitem{popat2023} K. M. Popat, K. R. Shingala, Some new results on energy of graphs with self-loops, \emph{J. Math. Chem.}, \textbf{61} (6) (2023) 1462--1469.

\bibitem{popat2024} K. M. Popat, K. R. Shingala, On equienergetic graphs and graph energy of some standard graphs with self loops, \emph{Proyecciones Journal of Mathematics}, \textbf{43} (5) (2024) 1269--1281.

\bibitem{preetha2023} U. P. Preetha, M. Suresh, E. Bonyah, On the spectrum, energy, and Laplacian energy of graphs with self-loops, \emph{Heliyon}, \textbf{9} (7) (2023) e17001.

\bibitem{haemers} W. H. Haemers, Seidel switching and graph energy, \emph{MATCH Commun. Math. Comput. Chem.} \textbf{68} (3) (2012) 653--659.    

\bibitem{akbari2020} S. Akbari, M. Einollahzadeh, M. M. Karkhaneei, M. A. Nematollahi, Proof of a conjecture on the Seidel energy of graphs, \emph{European J. Combin.}, \textbf{86} (2020) 103078.

\bibitem{einollahzadeh2022} M. Einollahzadeh, M. A. Nematollahi, An improved lower bound for the Seidel energy of trees, \emph{Discrete Appl. Math.}, \textbf{320} (2022) 381--386.

\bibitem{oboudi2016} M. R. Oboudi, Energy and Seidel energy of graphs, \emph{MATCH Commun. Math. Comput. Chem.}, \textbf{75} (2016) 291--303.  

\bibitem{oboudi2023} M. R. Oboudi, M. A. Nematollahi, Improving a lower bound for Seidel energy of graphs, \emph{MATCH Commun. Math. Comput. Chem.}, \textbf{89} (2) (2023) 489--502.

\bibitem{popat2019} S. K. Vaidya, K. M. Popat, \emph{Some New Results On Seidel Equienergetic Graphs,} \emph{KYUNGPOOK Math. J.}, \textbf{59} (2019) 335--340. 

\bibitem{nikiforov} V. Nikiforov, The energy of graphs and matrices, \emph{J. Math. Anal. Appl}. \textbf{326} (2007) 1472--1475.  
    
\bibitem{fan} K. Fan, Maximum properties and inequalities for the eigenvalues of completely continuous operators, \emph{Proc. Natl. Acad. Sci. USA} \textbf{37} (1951) 760--766. 

\bibitem{hornjohnson} R. A. Horn, C. R. Johnson, \emph{Matrix Analysis}, Cambridge University Press, Cambridge, 1985.

\bibitem{fiedler} M. Fiedler, Eigenvalues of nonnegative symmetric matrices, \emph{Linear Algebra and its Applications} \textbf{9} (1974) 119--142.
    
\bibitem{bro&hae} A. E. Brouwer, W. H. Haemers, \emph{Spectra of Graphs}, Springer, New York, 2011.  

  



   
    
\end{thebibliography}
\end{document}